\documentclass[a4paper,12pt]{amsart}
\usepackage{graphicx} 
\usepackage{inputenc}
\usepackage{tikz}
\usetikzlibrary{cd}
\usepackage{CJK}
\usepackage{geometry,url,color}
\usepackage{amsfonts,amssymb,amsmath,fancyhdr,mathrsfs}
\usepackage{amsthm,array,ragged2e,graphicx,enumitem}
\usepackage{booktabs,float,multirow,diagbox,longtable}
\usepackage{xcolor}
\usepackage[numbers,square]{natbib}

\definecolor{mycolor}{HTML}{C7EDCC}    

\allowdisplaybreaks
\justifying
\newtheorem{theorem}{Theorem}[section]
\newtheorem{definition}[theorem]{Definition}
\newtheorem{lemma}[theorem]{Lemma}
\newtheorem{proposition}[theorem]{Proposition}

\newtheorem{corollary}[theorem]{Corollary}
\newtheorem{remark}[theorem]{Remark}
\newtheorem{example}[theorem]{Example}

\newcommand{\Aut}{\mathrm{Aut}}
\newcommand{\Hom}{\mathrm{Hom}}

\newcommand{\cone}{\mathrm{cone}}

\newcommand{\mmod}{\mathrm{mod}}

\newcommand{\rep}{\mathrm{rep}}

\newcommand{\Coh}{\mathrm{Coh}}

\newcommand{\linp}[1]{\left\langle#1\right\rangle}
\newcommand{\udv}{\underline{\dim}}

\newcommand{\Ss}{\mathrm{SS}}

\newcommand{\rank}{\mathrm{rank}}
\newcommand{\add}{\mathrm{add}}
\newcommand{\proj}{\mathrm{proj}}

\title[Stability Conditions and HN Filtrations for Triangulated Categories]{
Stability Conditions and Harder-Narasimhan Filtrations for Triangulated Categories
}
\author[Wenyu Gao and Fan Xu]{
Wenyu Gao* and Fan Xu
}
\address{Department of Mathematical Sciences\\
Tsinghua University\\
Beijing 100084, P.~R.~China} \email{gwy23@mails.tsinghua.edu.cn (W. Gao)}

\address{Department of Mathematical Sciences\\
Tsinghua University\\
Beijing 100084, P.~R.~China} \email{fanxu@mail.tsinghua.edu.cn (F. Xu)}
\keywords{Derived Hall algebra, Triangulated category, Stability condition, Harder-Narasimhan filtration, Wall-crossing Formula}
	\thanks{$*$~Corresponding author.}

\begin{document}

\begin{abstract}

In this paper, we investigate the relationships between Harder-Narasimhan filtrations and derived Hall algebras. We extend several results from abelian categories to triangulated categories, including Reineke inversions, wall-crossing formulas, and Joyce's elements $\epsilon_\gamma$. The results in triangulated categories can be summarized via a diagram of the same form of that in abelian categories. As an application, we characterize all possibilities for stability conditions on $D^b(\rep A_2)$.
\end{abstract}
\maketitle
\section{Introduction}
    In \cite{Rei2002}, Reineke introduced the notion of HN filtration in representation categories of quivers, as an analog of a concept with the same name in the theory of moduli of vector bundles on curves. He established the relationship between stability functions and the Ringel-Hall algebras of representation categories. 
    
    Reineke's result was rewritten and summarized by Bridgeland and Toledano Laredo in \cite{BT2012} as the right part of the following diagram:
    \begin{center}
    \begin{tikzcd}
(\epsilon) \arrow[rr, "\exp", bend left] &  & (\delta) \arrow[ll, "\log", bend left] \arrow[rr, "\text{Reineke inversion}", bend left] &  & (\kappa) \arrow[ll, "\text{HN filtration}", bend left]
\end{tikzcd}
\end{center}
The precise definition of the symbols in the diagram will be introduced in Section 3. Roughly speaking, $(\kappa)$ and $(\delta)$ are characteristic functions, supported on all and semistable objects of a certain dimension vector, respectively.

In the left part of the diagram, the symbol $(\epsilon)$ denotes the elements defined by Joyce in \cite{Joy2007}. These are not characteristic functions in general, but they are supported on indecomposable elements. Thus, the diagram illustrates the relationships among all objects, semistable objects, and indecomposable objects.

Reineke also proved an identity \cite[Proposition 4.12]{Rei2002}, which is written as
\[1_\mathcal{A}=\prod^\leftarrow_{\phi\in(0,1]}\Ss^Z_\phi\]
in the notation of Bridgeland and Toledano Laredo. Although it does not appear in the diagram, it plays a major role in the theory. We will call it the wall-crossing formula in this paper. It characterizes the existence and uniqueness of the Harder–Narasimhan filtration for a certain object in $\mathcal{A}$, while implying that the descending product on the right-hand side of the identity is an invariant under the change of stability functions.
    
    Reineke's work was generalized by Joyce in \cite{Joy2007'} and \cite{Joy2008}, where Joyce raised the question \cite[Problem 7.1]{Joy2008} of introducing a parallel theory on triangulated categories. Following this question, Bridgeland \cite{Bri2002} introduced the notion of stability conditions on triangulated categories, which can be defined on bounded derived categories of abelian categories. Regarding Hall algebras, T\"{o}en \cite{Toe2006} constructed an associative algebra for each derived category, called the derived Hall algebra. This construction was later extended by Xiao and Xu \cite{XX2006} to triangulated categories satisfying certain conditions. They also provided a direct proof of the associativity of the derived Hall algebras.

    With the preparation above, we are now ready to generalize Reineke's work to triangulated categories satisfying both the conditions of Bridgeland and those of Xiao and Xu. The categories satisfying all the conditions do exist. For example, the bounded derived category of a $\Hom$-finite Krull-Schmidt abelian category satisfies all the conditions.

    In such a category, Wang and Chen \cite{WC2023} investigated the HN filtrations  using the language of derived Hall algebras. Their result is similar in form to that of Reineke. The present paper follows their approach and recovers the entirety of Reineke's theory. We summarize all our results again in a diagram that shares the same structure as the one by Bridgeland and Toledano Laredo:
    \begin{center}
    \begin{tikzcd}
(\epsilon) \arrow[rr, "\exp", bend left] &  & (\delta) \arrow[ll, "\log", bend left] \arrow[rr, "\text{Reineke inversion}", bend left] &  & (\kappa) \arrow[ll, "\text{HN filtration}", bend left]
\end{tikzcd}
\end{center}
where the meaning of each symbol may differ slightly from that in abelian categories, and we will define and examine them thoroughly in Section 4. We also present the wall-crossing formula within the context of triangulated categories:
\[1_\mathcal{D}=\prod^\leftarrow_{\phi\in\mathbb{R}}\Ss^\sigma_\phi.\]

As an application of our results, we determine all the stability conditions on the category $D^b(\rep A_2)$, classified by their hearts and the phases of objects in those hearts. We list all $12$ possibilities for stability conditions in a table at the end of this paper.

The result of this paper might be a partial answer to the question \cite[Problem 7.1]{Joy2008} raised by Joyce. However, providing a complete answer to this question remains a very challenging task.

\subsection{Organization of the Paper}
    This paper is organized as follows: 
    
    In Section 2, we review fundamental concepts related to quiver representations and stability functions in abelian categories. 
    
    In Section 3, we revisit the summary by Bridgeland and Toledano Laredo on Reineke’s work in \cite{BT2012}. 
    
    In Section 4, we generalize the theory to triangulated categories, which is the main contribution of this paper. 
    
    Finally, in Section 5, we present an example for the case of $A_2$, by determining all probabilities of stability conditions, hearts, wall-crossing formulas, and relations between them.

\subsection{Conventions and Notations}
    In this paper, every category is assumed to be additive, and every subcategory is assumed to be full. For a category $\mathcal{C}$, we write $X\in\mathcal{C}$ for that $X$ is an object in $\mathcal{C}$, and $\mathcal{B}\subseteq\mathcal{C}$ for that $\mathcal{B}$ is a subcategory of $\mathcal{C}$.

    For a category $\mathcal{C}$, we use $K^b(\mathcal{C})$ to denote the bounded homotopy category of $\mathcal{C}$. For an abelian category $\mathcal{A}$, we use $D^b(\mathcal{A})$ to denote the bounded derived category of $\mathcal{A}$.

    Let $\mathcal{S}$ be a subset of $\mathcal{C}$. We denote by $\add\mathcal{S}$ the additive hull of $\mathcal{S}$, namely the smallest additive subcategory of $\mathcal{C}$ containing all direct summands of finite direct sums of objects in $\mathcal{S}$. If $\mathcal{S}=\{X_1,\cdots,X_n\}$ is finite, we also write $\add(X_1,\cdots,X_n)=\add\mathcal{S}$. Let $\mathcal{A}$ be an abelian category, then the subcategory of $\mathcal{A}$ consisting of all projective objects in $\mathcal{A}$ is additive, which is denoted by $\proj\mathcal{A}$. 

    For a triangulated category $\mathcal{D}$, we use $[1]$ to denote the shift functor of $\mathcal{D}$. For a triangle
    \[X\rightarrow Y\rightarrow Z\rightarrow X[1],\]
    we may write it as
    \begin{center}
        \begin{tikzcd}
    X \arrow[rr] &                      & Y \arrow[ld] \\
                 & Z \arrow[lu, dashed] &             
    \end{tikzcd}
    \end{center}
    when necessary.
\section{Stability Functions on Abelian Categories}
In this section, we briefly review some fundamental concepts and results in abelian categories. Our primary focus will be on the representation categories of quivers. For a more detailed exploration, we refer the reader to \cite{BT2012} and \cite{Rei2002}. 

\subsection{Quivers and Representations}
First, we provide our conventions regarding quivers.
\begin{definition}
    A \textbf{quiver} is a quadruple $Q=(Q_0,Q_1,s,t)$, where $Q_0$ and $Q_1$ are two finite sets, called the sets of \textbf{vertices} and \textbf{arrows}, respectively, and $s,t:Q_1\rightarrow Q_0$ are two maps, known as the \textbf{source} and \textbf{target} maps, respectively.
\end{definition}
\begin{definition}
    A \textbf{path} of length $l$ in a quiver $Q$ is defined to be a sequence $a_1,\dots,a_l$ of arrows in $Q$, such that $t(a_i)=s(a_{i+1})$ for $1\leqslant i\leqslant l-1$. A path above is denoted by $a_l\dots a_1$. A path $a_l\dots a_1$ is called a \textbf{cycle} if $s(a_1)=t(a_l)$. A quiver with no cycles is said to be \textbf{acyclic}.
\end{definition}
Let $k$ be a field, we denote the \textbf{path algebra} of $Q$ over $k$ by $kQ$. It is the $k$-algebras with basis the set of all paths in $Q$, as an $k$-vector space, with multiplication defined by connecting paths. We have the canonical equivalence between the category $\rep_kQ$ of (finite dimensional) representations of $Q$ over $k$ and the category $\mmod kQ$ of (finite dimensional) left $kQ$-modules. It is well known that both categories are abelian $k$-categories. See \cite{ASS2006}. We will always identify them and denote them by $\mathcal{A}$.

\subsection{Slope Functions}
Let us fix an integral linear form $\Theta$ on $\mathbb{Z}Q_0$, that is, a $\mathbb{Z}$-linear function
\[\Theta:\mathbb{Z}Q_0\rightarrow\mathbb{Z}.\]
Then we obtain a function
\begin{align*}
    \mu:\mathbb{N}Q_0\backslash\{0\}&\rightarrow\mathbb{Q}\\
    d=(d_i)_{i\in I}&\mapsto{\frac{\sum_{i\in I}\Theta(d_i)}{\sum_{i\in I}d_i}},
\end{align*}
called the \textbf{slope function} associated to $\Theta$. For $M\in\mathcal{A}$, we always write
\[\mu(M)=\mu(\udv M).\]
\begin{definition}
    Let $\mu$ be as above. An object $0\neq M\in\mathcal{A}$ is said to be $\mu$-\textbf{semistable} if $\mu(U)\leqslant\mu(M)$ for any $0\neq U\subseteq M$. $M$ is said to be $\mu$-\textbf{stable} if $\mu(U)<\mu(M)$ for any $0\neq U\subsetneqq M$.
\end{definition}
The following three lemmas are all easy and well known. We refer the reader to \cite{Rei2002}.
\begin{lemma}
    Let $0\rightarrow L\rightarrow M\rightarrow N\rightarrow 0$ be a short exact sequence in $\mathcal{A}$. Then
    \begin{enumerate}
        \item $\mu(L)\leqslant\mu(M)\iff\mu(M)\leqslant\mu(N)\iff\mu(L)\leqslant\mu(N)$;
        \item $\min(\mu(L),\mu(N))\leqslant\mu(M)\leqslant\max(\mu(L),\mu(N))$;
        \item if $\mu(L)=\mu(M)=\mu(N)$, then $M$ is $\mu$-semistable if and only if $L$ and $N$ are both $\mu$-semistable.
    \end{enumerate}
\end{lemma}
\begin{lemma}
For $0\neq M\in\mathcal{A}$.
\begin{enumerate}
    \item If $M$ is $\mu$-stable, then $M$ is indecomposable.
    \item If $M$ is simple, then $M$ is $\mu$-stable.
\end{enumerate}
\end{lemma}
\begin{lemma}
    If $\mu(M)>\mu(N)$, then $\Hom_\mathcal{A}(M,N)=0$.
\end{lemma}
\subsection{Stability Functions}
Now consider the group homomorphism
\begin{align*}
    Z:K_0(\mathcal{A})=\mathbb{Z}Q_0&\rightarrow\mathbb{C}\\
    d_j&\mapsto1+\Theta(d_j)i.
\end{align*}
Again, we write $Z(M)$ for $Z(\udv M)$, then we have
\[Z(M)\in\mathbb{H}=\{e^{i\pi\phi}|\phi\in (0,1)\}\]
for any $0\neq M\in\mathcal{A}$. Now we can give the definition of stability functions.
\begin{definition}
    Let $\mathcal{A}$ be an abelian category (no need to be a representation category), and $K_0(\mathcal{A})$ be its Grothendieck group. A \textbf{stability function} on $\mathcal{A}$ is a group homomorphism $Z:K_0(\mathcal{A})\rightarrow\mathbb{C}$, such that $Z(M)\in\mathbb{H}\cup\mathbb{R}_-$ for $0\neq M\in\mathcal{A}$.
\end{definition}
    For each $0\neq M\in\mathcal{A}$ there is a unique $\phi(M)\in(0,1]$ such that $Z(M)=m_Me^{i\pi\phi(M)}$ for some $m_M\in\mathbb{R}_+$. We call $\phi(M)$ the \textbf{phase} of $M$ under $Z$.
\begin{definition}
    Let $Z$ be a stability function on $\mathcal{A}$. An object $0\neq M\in\mathcal{A}$ is said to be $Z$-\textbf{semistable} if $\phi(U)\leqslant\phi(M)$ for any $0\neq U\subseteq M$. $M$ is said to be $Z$-\textbf{stable} if $\phi(U)<\phi(M)$ for any $0\neq U\subsetneqq M$.
\end{definition}
The notion of stability functions is a generalization of the notion of slope functions, and similarly we have the following three lemmas.
\begin{lemma}\label{phase in ses}
    Let $0\rightarrow L\rightarrow M\rightarrow N\rightarrow 0$ be a short exact sequence in $\mathcal{A}$. Then
    \begin{enumerate}
        \item $\phi(L)\leqslant\phi(M)\iff\phi(M)\leqslant\phi(N)\iff\phi(L)\leqslant\phi(N)$;
        \item $\min(\phi(L),\phi(N))\leqslant\phi(M)\leqslant\max(\phi(L),\phi(N))$;
        \item if $\phi(L)=\phi(M)=\phi(N)$, then $M$ is $Z$-semistable if and only if $L$ and $N$ are both $Z$-semistable.
    \end{enumerate}
\end{lemma}
\begin{lemma}\label{stable,indecomposable,simple}
Let $\mathcal{A}$ be an abelian category and $0\neq M\in\mathcal{A}$.
\begin{enumerate}
    \item If $M$ is $Z$-stable, then $M$ is indecomposable.
    \item If $M$ is simple, then $M$ is $Z$-stable.
\end{enumerate}
\end{lemma}
\begin{lemma}\label{hom to phi big}
    If $\phi(M)>\phi(N)$, then $\Hom_\mathcal{A}(M,N)=0$.
\end{lemma}
\begin{example}
    Let $\mathcal{A}=\Coh(\mathbb{P}^1)$ be the category of coherent sheaves over $\mathbb{P}^1$, which is an abelian category. The function
    \[Z(E)=-\deg E+i\rank(E)\]
    defines a stability function $Z$ on $\mathcal{A}$. It is easy to see that $Z(K_0(\mathcal{A}))\not\subseteq\mathbb{H}$ since $Z(k_x)=-1$ for any degree $1$ point $x\in\mathbb{P}^1$.
\end{example}
\subsection{HN Filtrations}
We now introduce the concept of HN filtrations.
\begin{definition}\label{HN filtration}
    Let $\mathcal{A}$ be an abelian category with a stability function $Z$. For $0\neq M\in\mathcal{A}$, we call the filtration
    \[0=M_0\subseteq M_1\subseteq\dots\subseteq M_s=M,\]
    such that $A_i:=M_i/M_{i-1}$ is $Z$-semistable, $1\leqslant i\leqslant s$, with
    \[\phi(A_1)>\phi(A_2)>\dots>\phi(A_s)\]
    the \textbf{Harder-Narisimhan filtration} or \textbf{HN filtration} for short. It is unique up to isomorphism. The tuple of dimension vectors $(\udv A_1,\dots,\udv A_s)$ is called the \textbf{HN type} of $M$.

    The stability function $Z$ is said to \textbf{have the HN property} if any $0\neq M\in\mathcal{A}$ possesses an HN filtration.
\end{definition}
Next, we will mainly consider the finite dimensional module categories over a finite dimensional algebra, or equivalently, finite dimensional representation categories of quivers (with relations). It is easy to see that those categories are all of finite length, and we have the following proposition:
\begin{proposition}
    Let $\mathcal{A}$ be an abelian category of finite length. Then every stability function $Z$ on $\mathcal{A}$ has the HN property.
\end{proposition}
\begin{proof}
    It is a natural consequence of \cite[Proposition 2.4]{Bri2002}.
\end{proof}
\begin{example}\label{A2}
    Let $\mathcal{A}=\rep_k(A_2)$, where $A_2:1\rightarrow 2$. Then $\mathcal{A}$ has three indecomposable objects:
    \[S_1:k\rightarrow0,S_2:0\rightarrow k, P_1:k\rightarrow k\]
    up to isomorphism, with a short exact sequence
    \[0\rightarrow S_2\rightarrow P_1\rightarrow S_1\rightarrow0.\]
    The Grothendieck group $K_0(\mathcal{A})=\mathbb{Z}^2$ is the free abelian group generated by $\udv S_1=(1,0)$ and $\udv S_2=(0,1)$. So a stability function on $\mathcal{A}$ is uniquely determined by an assignment for $Z(S_1)$ and $Z(S_2)$. 
    
    By \textbf{Lemma \ref{stable,indecomposable,simple}}(2), $S_1$ and $S_2$ are always $Z$-stable, hence $Z$-semistable, no matter how $Z$ change. But for $P_1$, the question is a little more difficult. The only nonzero proper subobject (up to isomorphism) of $P_1$ is $S_2$. So $P_1$ is $Z$-semistable if and only if $\phi(S_2)\leqslant \phi(P_1)$ and hence if and only if $\phi(S_2)\leqslant \phi(S_1)$ by \textbf{Lemma \ref{phase in ses}}(1). So the HN filtration of $P_1$ is $0\subseteq P_1$ when $\phi(S_2)\leqslant\phi(S_1)$ and $0\subseteq S_2\subseteq P_1$ when $\phi(S_2)>\phi(S_1)$.
\end{example}
In the case of \textbf{Example 
\ref{A2}}, we observe that when the stability function (Z) changes, differences emerge on the two ``sides" of the ``wall" $\phi(S_2)=\phi(S_1)$. This phenomenon is also present in general cases. However, during the change of the stability function, certain properties remain invariant. Our goal is to characterize some of them.
\section{Ringel-Hall Algebras and Wall-Crossing Formulas}
\subsection{Ringel-Hall Algebras}
For an abelian category $\mathcal{A}$, there are many ways to define its Hall
algebra. Here we follow the choice of \cite{BT2012} and \cite{Rei2002}.

Let $k=\mathbb{C}$ and $\mathcal{H}(\mathcal{A})$ be the $\mathbb{Q}$-vector space with basis the set of isomorphism classes of objects in $\mathcal{A}$. Now we define a multiplication on it to make it an associative $\mathbb{Q}$-algebra.

For $M_1,\dots,M_r,X\in\mathcal{A}$, the set of filtrations
\[V(M_1,\dots,M_r;X)=\{0=X_0\subseteq X_1\subseteq\dots\subseteq X_r=X|X_i/X_{i-1}\cong A_i,1\leqslant i\leqslant r\}\]
has the structure of a complex variety. We denote its Euler characteristic number (under the analytic topology) by $\chi(V(M_1,\dots,M_r;X))$. The following theorem was first proved by Schofield and a detailed study of it by Reidtmann can be found in \cite{Rie1994}.
\begin{theorem}
    The operation
    \[u_{[M]}*u_{[N]}=\sum_{[L]}\chi(V(M,N;L))u_{[L]}\]
    defines an associative multiplication on $\mathcal{H}(\mathcal{A})$ with the identity element $1=u_{[0]}$. The resulting associative $\mathbb{Q}$-algebra $(\mathcal{H}(\mathcal{A}),*,1)$ is called the \textbf{Ringel-Hall algebra} of $\mathcal{A}$.
\end{theorem}
We have another point of view on the elements in $\mathcal{H}(\mathcal{A})$. For $f=\sum_{[M]}f_{[M]}u_{[M]}\in\mathcal{H}(\mathcal{A})$, we may regard $f$ as a function from the set of isomorphism classes of objects in $\mathcal{A}$ to the complex field via
\[f([M]):=f_{[M]}.\]
So $\mathcal{H}(\mathcal{A})$ can be regarded as the set of functions on isomorphism classes of objects in $\mathcal{A}$ that vanish on all but finitely many points.
It is the viewpoint taken by Bridgeland and Toledano Laredo \cite{BT2012}, and we will maintain it throughout this section.

\subsection{Characteristic Functions}
In $\mathcal{H}(\mathcal{A})$, we define some important functions, using the same notation as in \cite{BT2012} and differing from those in \cite{Rei2002} and \cite{WC2023}.

Let us fix a stability function on $\mathcal{A}$. For $\gamma\in K_0(\mathcal{A})$, we define
\begin{align*}
    \kappa_\gamma(M)&=\begin{cases}
        1,&\udv M=\gamma,\\
        0,&\udv M\neq\gamma,
    \end{cases},\\
    \delta^Z_\gamma(M)&=\begin{cases}
        1,&\udv M=\gamma\text{ and $M$ is $Z$-semistable},\\
        0,&\text{otherwise}.
    \end{cases}
\end{align*}
These can be realized as characteristic functions supported on all objects and $Z$-semistable objects, respectively, with dimension vector $\gamma$. By \cite{Rei2002}, we have the following result.
\begin{theorem}\label{reciprocity}
For any $0\neq\gamma\in K_0(\mathcal{A})$, we have
    \begin{align*}
    \kappa_\gamma&=\sum_{\begin{subarray}{c}
        \\ \gamma_1+\dots+\gamma_n=\gamma\\
        \phi(\gamma_1)>\dots>\phi(\gamma_n)
        \end{subarray}}\delta^Z_{\gamma_1}*\dots*\delta^Z_{\gamma_n},\\
    \delta^Z_\gamma&=\sum_{\begin{subarray}{c}
        \\ \gamma_1+\dots+\gamma_n=\gamma\\
        \phi(\gamma_1+\dots+\gamma_i)>\phi(\gamma),1\leqslant i\leqslant n-1
        \end{subarray}}(-1)^{n-1}\kappa_{\gamma_1}*\dots*\kappa_{\gamma_n}.
    \end{align*}
    The second identity is called the \textbf{Reineke inversion}.
\end{theorem}
    The proof of the first identity follows immediately from the uniqueness of the HN filtration. The second was proved by Reineke in \cite{Rei2002}.   
    
The following subalgebra of $\mathcal{H}(\mathcal{A})$ is very important.
\begin{definition}
    The subalgebra
    \[C(\mathcal{A})=\linp{\kappa_\gamma\middle|\gamma\in K_0(\mathcal{A})}\]
    of $\mathcal{H}(\mathcal{A})$ is called the \textbf{composition subalgebra} of $\mathcal{H}(\mathcal{A})$.
\end{definition}
We can then deduce from \textbf{Theorem \ref{reciprocity}} that:
\begin{corollary}[{\cite[Theorem 4.9]{Rei2002}}]
    For any stability function $Z$ on $\mathcal{A}$ and $\gamma\in K_0(\mathcal{A})$, we have $\delta^Z_\gamma\in C(\mathcal{A})$.
\end{corollary}
\subsection{Comultiplication}
The importance of the composition algebras is that it has an structure of bialgebra and hence the set of primitive elements forms a Lie subalgebra.

Let us to introduce it for several steps, follow \cite{BT2012}. We first define a map $\Delta:\mathcal{H}(\mathcal{A})\rightarrow\mathcal{H}(\mathcal{A}\times\mathcal{A})$ via
\[\Delta(f)(M,N)=f(M\oplus N).\]
We have the following important result.
\begin{theorem}[{\cite[Theorem 4.3]{BT2012}}]
\begin{enumerate}
    \item The image of $C(\mathcal{A})$ under $\Delta$ is contained in the tensor product
    \[C(\mathcal{A})\otimes C(\mathcal{A})\subseteq \mathcal{H}(\mathcal{A})\otimes\mathcal{H}(\mathcal{A})\subseteq\mathcal{H}(\mathcal{A\times A}),\]
    where the second inclusion is understood as
    \[(f\otimes g)(M,N)=f(M)g(N).\]
    \item The restriction of $\Delta$ on $C(\mathcal{A})$ can be explicitly written as
    \[\Delta(\kappa_\gamma)=\sum_{\gamma_1+\gamma_2=\gamma}\kappa_{\gamma_1}\otimes\kappa_{\gamma_2.}\]
    \item Define $\eta(f)=f(0)$, then $(C(\mathcal{A}),*,1,\Delta,\eta)$ is a cocommutative bialgebra.
\end{enumerate}
\end{theorem}
We say that an element $p\in C(\mathcal{A})$ is \textbf{primitive} if $\Delta(p)=p\otimes1+1\otimes p$. The following lemma shows the relationship between primitive elements in $C(\mathcal{A})$ and indecomposable objects in $\mathcal{A}$.
\begin{lemma}[\cite{Rin1992}]\label{primitive}
    Let $\mathcal{S}$ be a Krull-Schimidt category and $C(\mathcal{S})$ be the vector space with basis the isomorphism classes of objects in $\mathcal{S}$. Let $\Delta$ be a comultiplication defined as above, then an element $p\in C(\mathcal{S})$ is primitive if and only if $p$ is supported on indecomposable objects in $\mathcal{S}$.
\end{lemma}
Let $\mathfrak{n}(\mathcal{A})$ be the subset of $C(\mathcal{A})$ that consists of all primitive elements. It is a Lie subalgebra of $C(\mathcal{A})$ by \cite{Rie1994} and \cite{Rin1992}, called the \textbf{Ringel-Hall Lie algebra}.

\subsection{Wall-crossing Formulas}
Let $\hat{C}(\mathcal{A})$ and $\hat{\mathfrak{n}}(\mathcal{A})$ be the completion of $C(\mathcal{A})$ and $\mathfrak{n}(\mathcal{A})$, respectively, in the meaning of \cite{BT2012}. Denote by $\hat{N}(\mathcal{A})$ the subset of $\hat{C}(\mathcal{A})$ that consists of all \textbf{grouplike elements}, namely the elements $g$ satisfying $\Delta(g)=g\otimes g$. Then the exponential map
\[\exp x:=\sum_{n\geqslant0}\frac{x^n}{n!}\]
gives a bijection from $\hat{\mathfrak{n}}(\mathcal{A})$ to $\hat{N}(\mathcal{A})$, with the inverse
\[\log y=\sum_{n\geqslant1}\frac{(-1)^{n-1}}{n}(y-1)^n.\]
$\hat{N}(\mathcal{A})$ is a subgroup of the unit group $\hat{C}(\mathcal{A})^\times$ of $\hat{C}(\mathcal{A})$, called the \textbf{Ringel-Hall group} of $\mathcal{A}$. It is a pro-unipotent Lie group with Lie algebra $\hat{\mathfrak{n}}(\mathcal{A})$.

We now define the following elements in $\hat{C}(\mathcal{A})$. For $\phi\in (0,1]$, set
\[\Ss^Z_\phi=1+\sum_{\gamma\in K_0(\mathcal{A}), \phi(\gamma)=\phi}\delta^Z_\gamma.\]
Then we can write the famous wall-crossing formula.
\begin{theorem}[\cite{Rei2002}]
    Let $1_\mathcal{A}$ be the element
    \[1_\mathcal{A}=\sum_{\gamma\in K_0(\mathcal{A})}\kappa_\gamma.\]
    Then,
    \[1_\mathcal{A}=\prod^{\leftarrow}_{\phi\in(0,1]}\Ss^Z_\phi.\]
    Moreover, the elements $1_\mathcal{A}$ and $\Ss^Z_\phi,\phi\in(0,1]$ are all grouplike, hence belong to $\hat{N}(\mathcal{A})$.
\end{theorem}
In the identity above, the left hand side $1_\mathcal{A}$ does not depend on the choice of the stability function $Z$, while any single $\Ss^Z_\phi$ does. Thus, the descending product 
\[\prod^{\leftarrow}_{\phi\in(0,1]}\Ss^Z_\phi\]
is an invariant under the changing of stability functions. For example, let us see the $A_2$ case again.
\begin{example}\label{A2,abelian}
    In \textbf{Example \ref{A2}}, we have determined all $Z$-semistable objects in every condition. When $\phi(S_2)>\phi(S_1)$, the only $Z$-semistable objects in $\mathcal{A}$ are of the form $S_1^{\oplus n}$ or $S_2^{\oplus n}$, $n\geqslant1$. So their phase must be $\phi(S_1)$ or $\phi(S_2)$, respectively, and hence the descending product is
    \[\prod^\leftarrow_{\phi\in(0,1]}\Ss^Z_\phi=\Ss^Z_{\phi(S_2)}*\Ss^Z_{\phi(S_1)}.\]
    If we set for every indecomposable $0\neq M\in\mathcal{A}$ that
    \[\Phi_M=\sum_{n=0}^\infty u_{[M^{\oplus n}]},\]
    then we have
    \[1_\mathcal{A}=\prod^\leftarrow_{\phi\in(0,1]}\Ss^Z_\phi=\Phi_{S_2}*\Phi_{S_1}.\]
    When $\phi(S_2)=\phi(S_1)=\phi$, every nonzero object in $\mathcal{A}$ is $Z$-semistable with phase $\phi$. So the descending product is trivially $1_\mathcal{A}$. When $\phi(S_2)<\phi(S_1)$, the $Z$-semistable objects in $\mathcal{A}$ can have three different types, $S_1^{\oplus n},S_2^{\oplus n}$ and $P_1^{\oplus n}$. By the same argument as in the first case, we obtain
    \[1_\mathcal{A}=\prod^\leftarrow_{\phi\in(0,1]}\Ss^Z_\phi=\Phi_{S_1}*\Phi_{P_1}*\Phi_{S_2}.\]
    Combining the two identity above yields
    \[\Phi_{S_2}*\Phi_{S_1}=\Phi_{S_1}*\Phi_{P_1}*\Phi_{S_2},\]
    which is known as the pentagon identity.
\end{example}
\subsection{Joyce's Element}
Now for $\gamma\in K_0(\mathcal{A})^+$, we define \textbf{Joyce's element} \cite{Joy2007}.
\[\epsilon^Z_\gamma=\sum_{\begin{subarray}{c}
    \\ \gamma_1+\dots+\gamma_n=\gamma\\ \phi(\gamma_i)=\phi(\gamma),1\leqslant i\leqslant n
\end{subarray}}\frac{(-1)^{n-1}}{n}\delta^Z_{\gamma_1}*\dots*\delta^Z_{\gamma_n}\in C(\mathcal{A}).\]
By calculating using the definition of the logarithm function, we obtain
\begin{align*}
    \log\Ss^Z_\phi&=\log\left(1+\sum_{ \phi(\gamma)=\phi}\delta^Z_\gamma\right)\\
    &=\sum_{n\geqslant1}\frac{(-1)^{n-1}}{n}\left(\sum_{\phi(\gamma)=\phi}\delta^Z_\gamma\right)^n\\
    &=\sum_{n\geqslant1}\frac{(-1)^{n-1}}{n}\sum_{\phi(\gamma_1)=\dots=\phi(\gamma_n)=\phi}\delta^Z_{\gamma_1}*\dots*\delta^Z_{\gamma_n}\\
    &=\sum_{\gamma\in K_0(\mathcal{A})^+,\phi(\gamma)=\phi}\epsilon^Z_\gamma
\end{align*}
\[,\]
or equivalently
\[\Ss^Z_\phi=\exp\left(\sum_{\gamma\in K_0(\mathcal{A})^+,\phi(\gamma)=\phi}\epsilon^Z_\gamma\right).\]
Consequently, $\epsilon_\gamma$ is primitive and hence supported on indecomposable objects.

\subsection{Summary: A Diagram by Bridgeland and Toledano Laredo}
In \cite{BT2012}, Bridgeland and Toledano Laredo draw a diagram to visualize the relationships between the three kinds of elements $(\kappa),(\delta)$ and $(\epsilon)$,
\begin{center}
    \begin{tikzcd}
(\epsilon) \arrow[rr, "\exp", bend left] &  & (\delta) \arrow[ll, "\log", bend left] \arrow[rr, "\text{Reineke inversion}", bend left] &  & (\kappa) \arrow[ll, "\text{HN filtration}", bend left]
\end{tikzcd}
\end{center}
which is the best summary of the theory introduced in this section.
\section{Generalizations to Triangulated Categories}
Throughout this chapter, we assume that $\mathcal{D}$ is a triangulated category over a field $k$ satisfying the following two conditions:
\begin{enumerate}
    \item $\mathcal{D}$ is Krull-Schmidt;
    \item $\mathcal{D}$ is \textbf{left homologically finite}, namely $\sum_{i\geqslant0}\dim_k\Hom_\mathcal{D}(X[i],Y)<+\infty$ for any $X,Y\in\mathcal{D}$.
\end{enumerate}
\begin{remark}
    These two conditions are, in fact, equivalent to the three conditions required in \cite{XX2006}, since the condition of left homological finiteness implies the condition of $\Hom$-finiteness.
\end{remark}
\begin{example} Let $\mathcal{A}$ be a $\Hom$-finite Krull-Schmidt abelian category.
    \begin{enumerate}
        \item The bounded derived category $D^b(\mathcal{A})$ satisfies our two conditions.
        \item The $n$-periodic derived category $D^b(A)/[n]$ does not satisfy our second condition. In particular, the root category $R(\mathcal{A})=D^b(\mathcal{A})/[2]$ does not satisfy our second condition.
    \end{enumerate}
\end{example}
For $M\in\mathcal{D}$, we denote by $\udv M$ the corresponding element of $M$ in $K_0(\mathcal{D})$.
\subsection{Derived Hall Algebras for Triangulated Categories}
We first introduce a structure of Hall algebra defined in \cite{XX2006}.

Let $k$ be a finite field with $q$ elements. For $L,M,N\in\mathcal{D}$, we define
\[(L,N)_M=\{f\in\Hom_\mathcal{D}(L,N)|\cone f\cong M\},\]
\[\{L,N\}=\prod_{i>0}|\Hom_{\mathcal{D}}(L[i],N)|^{(-1)^i},\]
and
\[F_{MN}^L=\frac{|(L,N)_{M[1]}|}{|\Aut N|}\frac{\{L,N\}}{\{N,N\}}\in\mathbb{Q}.\]
\begin{theorem}[\cite{XX2006}]
    Let $\mathcal{H}(\mathcal{D})$ be the $\mathbb{Q}$-vector space with basis the isomorphism classes of objects in $\mathcal{D}$. Define a binary operation $*$ on the basis of $\mathcal{H}(\mathcal{D})$:
    \[u_{[M]}*u_{[N]}=\sum_{[L]}F_{MN}^Lu_{[L]},\]
    and linearly span it to the whole space $\mathcal{H}(\mathcal{D})$. Then $(\mathcal{H}(\mathcal{D}),*)$ is an associated $\mathbb{Q}$-algebra with unit $u_{[0]}$.
\end{theorem}
We can also view an element $f$ in $\mathcal{H}(\mathcal{A})$ as a function on isomorphism classes by defining $f(M)$ to be the coefficient of $u_{[M]}$ in the expression of $f$ as a linear combination of all $u_{[M]}$. This approach allows us to maintain the perspective and notation introduced in Section 3.
\subsection{Stability Conditions}
We follow the introduction in \cite{Bri2002}. We point out that it is possible for a triangulated category (for example, a periodic one) to have no stability condition in the sense of Bridgeland. But in many cases, such as for bounded derived categories of abelian categories, stability conditions exist.
\begin{definition}
    Let $\mathcal{D}$ be a triangulated category. A \textbf{stability condition} on $\mathcal{D}$ is a pair $\sigma=(Z,\mathcal{P})$, where
    \begin{itemize}
        \item $Z:K_0(\mathcal{D})\rightarrow\mathbb{C}$ is a group homomorphism, called the \textbf{central charge};
        \item $\mathcal{P}=\{\mathcal{P}(\phi)|\phi\in\mathbb{R}\}$ is a family of full additive extension-closed subcategories of $\mathcal{D}$,
    \end{itemize}
    satisfying the following conditions:
    \begin{enumerate}
        \item If $M\in\mathcal{P}(\phi)$, then $Z(M)=m_Me^{i\pi\phi}$ for some $m_M\in\mathbb{R}_+$.
        \item $\mathcal{P}(\phi+1)=\mathcal{P}(\phi)[1]$ for any $\phi\in\mathbb{R}$.
        \item For $A_i\in\mathcal{P}(\phi_i),i=1,2$ with $\phi_1>\phi_2$, we have $\Hom_\mathcal{D}(A_1,A_2)=0$.
        \item (HN filtration) For any $0\neq M\in\mathcal{D}$, there is a sequence of triangles
        \begin{center}
            \begin{tikzcd}[sep=tiny]
0=M_0 \arrow[rr] &                        & M_1 \arrow[rr] \arrow[ld] &                        & M_2 \arrow[r] \arrow[ld] & \cdots \arrow[r] & M_{n-1} \arrow[rr] &                        & M_n=M \arrow[ld] \\
                 & A_1 \arrow[lu, dashed] &                           & A_2 \arrow[lu, dashed] &                          &                  &                    & A_n \arrow[lu, dashed] &                 
\end{tikzcd}
        \end{center}
        such that each $A_i$ belongs to $\mathcal{P}(\phi_i)$ for some $\phi_i$, $1\leqslant i\leqslant n$, with $\phi_1>\cdots>\phi_n$. This sequence is called an \textbf{Harder-Narasimhan filtration}, or an \textbf{HN filtration} for short. It is unique up to isomorphism by \cite{Bri2002}. Thus we can define
        \[\phi^+(M)=\phi_1,\phi^-(M)=\phi_n.\]
        We call the $n$-tuple of dimension vectors $(\udv A_1,\dots,\udv A_n)$ the \textbf{HN type} of $M$.
    \end{enumerate}
\end{definition}

An object $0\neq M\in\mathcal{D}$ is said to be $\sigma$-\textbf{semistable} with \textbf{phase} $\phi$ if $M\in\mathcal{P}(\phi)$, or equivalently, $\phi^+(M)=\phi^-(M)=\phi$. Moreover, $M$ is said to be $\sigma$-\textbf{stable} if $M$ is a simple object in the category $\mathcal{P}(\phi)$.

The following lemma follows immediately from the definition.
\begin{lemma}\label{stability under shift}
    Let $0\neq M\in\mathcal{D}$ and $n\in\mathbb{Z}$, then $M$ is $\sigma$-semistable if and only if $M[n]$ is $\sigma$-semistable. Moreover, $M$ is $\sigma$-stable if and only if $M[n]$ is $\sigma$-stable.
\end{lemma}


For an interval $I$, we define a subcategory 
\[\mathcal{P}(I)=\{0\}\cup\{0\neq M\in\mathcal{D}|\phi^+(M),\phi^-(M)\in I\}\]
of $\mathcal{D}$. Then the pair of subcategories $(\mathcal{P}(-\infty,\phi],\mathcal{P}(\phi,+\infty))$ for any $\phi\in\mathbb{R}$ is a bounded t-structure on $\mathcal{D}$, with heart $\mathcal{P}(\phi,\phi+1]$. In pariticular, when $\phi=0$, the heart $\phi(0,1]$ of the t-structure $(\mathcal{P}(-\infty,0],\mathcal{P}(1,+\infty))$ is called the \textbf{heart} of $\sigma$.

Bridgeland proved the following important theorem, which shows the relationship between a stability function on an abelian category and a stability condition on a triangulated category.
\begin{theorem}[{\cite[Proposition 5.3]{Bri2002}}]\label{heart}
    To give a stability condition on a triangulated category $\mathcal{D}$ is equivalent to giving a bounded t-structure on $\mathcal{D}$ and a stability function having the HN property (see \textbf{Definition \ref{HN filtration}}) on its heart.
\end{theorem}
\begin{corollary}
    Let $\mathcal{A}$ be an abelian category of finite length. There exists a one-to-one correspondence between stability conditions on $D^b(\mathcal{A})$ with $\mathcal{A}$ as its heart and stability functions on $\mathcal{A}$. 
\end{corollary}
We have the following easy lemma.
\begin{lemma}\label{semistable vs direct sum}
    Let $M,N\in\mathcal{D}$, then $M\oplus N\in\mathcal{P}(\phi)$ if and only if $M\in\mathcal{P}(\phi)$ and $N\in\mathcal{P}(\phi)$.
\end{lemma}
\begin{proof}
    The if part follows from the assumption of $\mathcal{P}(\phi)$ to be additive. Now let us show the only if part. It has nothing to show when one of $M$ and $N$ is zero, so we will assume that $M,N\neq 0$.

    Consider the triangle
    \[M\xrightarrow{i} M\oplus N\xrightarrow{p} N\rightarrow M[1].\]
    Let $\phi_1=\phi^+(M),\phi_2=\phi^-(N)$, so there exist $M^+\in\mathcal{P}(\phi_1),N^-\in\mathcal{P}(\phi_2)$ with nonzero morphisms $M^+\rightarrow M$ and $N\rightarrow N^-$. Therefore, the morphisms
    \[M^+\rightarrow M\xrightarrow{i}M\oplus N, M\oplus N\xrightarrow{p}N\rightarrow N^-\]
    are nonzero, so $\phi_1\leqslant\phi^+(M\oplus N)=\phi=\phi^-(M\oplus N)\leqslant\phi_2$. That is,
    \[\phi^+(M)\leqslant\phi\leqslant\phi^-(N).\]
    The same examination on the triangle
    \[N\rightarrow M\oplus N\rightarrow M\rightarrow N[1]\]
    yields
    \[\phi^+(N)\leqslant\phi\leqslant\phi^-(M).\]
    Thus $\phi^+(M)=\phi^-(M)=\phi=\phi^+(N)=\phi^-(N)$ and we finish the proof.
\end{proof}
\begin{corollary}\label{Stable is indecomposable}
    If $0\neq M\in\mathcal{D}$ is $\sigma$-stable, then $M$ is indecomposable.
\end{corollary}
\begin{proof}
    Suppose that $M=M_1\oplus M_2$. By assumption, $M\in\mathcal{P}(\phi)$ for some $\phi\in\mathbb{R}$. Thus, by \textbf{Lemma \ref{semistable vs direct sum}}, we have $M_1,M_2\in\mathcal{P}(\phi)$ and hence are subobjects of $M$ in $\mathcal{P}(\phi)$. Since $M$ is $\sigma$-stable, it is simple in $\mathcal{P}(\phi)$, so one of $M_1$ and $M_2$ is zero and the other is $M$.
\end{proof}
\begin{lemma}\label{phase increasing}
    Let $M\rightarrow X\rightarrow N\rightarrow M[1]$ be a triangle in $\mathcal{D}$ with $M,X,N$ are all $\sigma$-semistable, then $\phi(M)\leqslant\phi(X)\leqslant\phi(N)$.
\end{lemma}
\begin{proof}
    If $\phi(M)>\phi(N)$, then
    \begin{center}
        \begin{tikzcd}
0 \arrow[rr] &                      & M \arrow[rr] \arrow[ld] &                      & X \arrow[ld] \\
             & M \arrow[lu, dashed] &                         & N \arrow[lu, dashed] &             
\end{tikzcd}
    \end{center}
    is an HN filtration of $X$, which contradicts the assumption that $X$ is $\sigma$-semistable. So $\phi(M)\leqslant\phi(N)$.

    Note that
    \[X\in\mathcal{P}(\phi(N))*\mathcal{P}(\phi(M))\subseteq\mathcal{P}[\phi(M),\phi(N)],\]
    so $\phi(X)\in[\phi(M),\phi(N)]$.
\end{proof}
\subsection{Characteristic Functions}
In this subsection, we will define the two types of characteristic functions $(\kappa)$ and $(\delta)$, which will be similar to those introduced in Section 3. We follow the description in \cite{WC2023}, although we employ significantly different notations.

Firstly, for $0\neq\gamma\in K_0(\mathcal{D})$, we define
\[\kappa_\gamma(M)=\begin{cases}
    1,&\udv M=\gamma,\\
    0,&\udv M\neq\gamma,
\end{cases}\]
together with
\[\kappa_0(M)=\begin{cases}
    1,&\udv M=0\text{ and }M\neq0,\\
    0,&\udv M\neq\gamma\text{ or }M=0.
\end{cases}\]
There is a little different from the case in abelian categories at the dimension vector $\gamma=0$, because in a triangulated category a nonzero object may have zero dimension vector, such as  $X\oplus X[1]$. After we fix a stability condition $\sigma=(Z,\mathcal{P})$ on $\mathcal{D}$, we can define
\[\delta^\sigma_\gamma(M)=\begin{cases}
    1,&\udv M=\gamma\text{ and $M$ is $\sigma$-semistable},\\
    0,&\text{otherwise},
\end{cases}\]
for $0\neq\gamma\in K_0(\mathcal{D})$.
We aim to reconstruct the diagram
\begin{center}
    \begin{tikzcd}
(\epsilon) \arrow[rr, "\exp", bend left] &  & (\delta) \arrow[ll, "\log", bend left] \arrow[rr, "\text{Reineke inversion}", bend left] &  & (\kappa) \arrow[ll, "\text{HN filtration}", bend left]
\end{tikzcd}
\end{center}
within the context of abelian categories. Let us examine it arrow by arrow.
\subsection{The Right Lower Arrow: HN Filtration}
This question has been thoroughly examined in \cite{WC2023}. We now present it using our notation and lay the groundwork for the subsequent subsections.

For $\gamma\in K_0(\mathcal{D})$, consider the set
\[\mathbf{R}_\gamma=\{0\neq M\in\mathcal{D}|\udv M=\gamma\}.\]
With our notation, $\kappa_\gamma$ is the characteristic function supported on $\mathbf{R}_\gamma$. For an $s$-tuple of dimension vectors $\gamma^*=(\gamma_1,\dots,\gamma_s)$, we define
another set
\[\mathbf{R}_{\gamma^*}^{HN}=\{0\neq M\in\mathcal{D}|\text{the HN type of $M$ is equal to $\gamma^*$}\},\]
with $\chi^\sigma_{\gamma^*}$ as characteristic function supported on it. In particular, we set $\mathbf{R}_\gamma^{ss}=\mathbf{R}_{(\gamma)}^{HN}$, whose characteristic function is $\delta^\sigma_\gamma$.

By the existence and uniqueness of HN filtrations for all nonzero objects, we have
\[\mathbf{R}_\gamma=\bigsqcup_{|\gamma^*|=\gamma}\mathbf{R}_{\gamma^*}^{HN},\]
where $|\gamma^*|=\gamma_1+\dots+\gamma_s$ for $\gamma^*=(\gamma_1,\dots,\gamma_s)$. By taking characteristic functions on both sides, we deduce the identity
\[\kappa_\gamma=\sum_{|\gamma*|=\gamma}\chi^\sigma_{\gamma^*}.\]

Wang and Chen proved the following lemma.
\begin{lemma}[{\cite[Proposition 3.6]{WC2023}}]\label{Wang-Chen}
    Let $\gamma^*=(\gamma_1,\dots,\gamma_n)$, then
    \[\chi^\sigma_{\gamma^*}=\delta^\sigma_{\gamma_1}*\dots*\delta^\sigma_{\gamma_n}.\]
\end{lemma}
\begin{proof}
    For $X\in\mathbf{R}_{\gamma^*}^{HN}$, write down its unique HN filtration:
    \begin{center}
        \begin{tikzcd}[sep=small]
0=X_0 \arrow[rr] &                        & X_1 \arrow[rr] \arrow[ld] &                        & X_2 \arrow[r] \arrow[ld] & \cdots \arrow[r] & X_{n-1} \arrow[rr] &                        & X_n=X \arrow[ld] \\
                 & A_1 \arrow[lu, dashed] &                           & A_2 \arrow[lu, dashed] &                          &                  &                    & A_n \arrow[lu, dashed] &                 
        \end{tikzcd}
    \end{center}
    By the definition of multiplication, we have
    \[u_{[X]}=\prod_{i=1}^n F_{X_{i-1}A_i}^{X_i}u_{[X_1]}*\cdots*u_{[X_n]}.\]
    The rest is the following lemma. We refer the reader to the proof of \cite[Proposition 3.6]{WC2023}.
    \begin{lemma}\label{WC}
        Let $0\neq X\in\mathcal{D}$ with HN filtration
        \begin{center}
        \begin{tikzcd}[sep=small]
0=X_0 \arrow[rr] &                        & X_1 \arrow[rr] \arrow[ld] &                        & X_2 \arrow[r] \arrow[ld] & \cdots \arrow[r] & X_{n-1} \arrow[rr] &                        & X_n=X \arrow[ld] \\
                 & A_1 \arrow[lu, dashed] &                           & A_2 \arrow[lu, dashed] &                          &                  &                    & A_n \arrow[lu, dashed] &                 
        \end{tikzcd}
    \end{center}
    Then
    \[\prod_{i=1}^n F_{X_{i-1}A_i}^{X_i}=1.\]
    \end{lemma}
\end{proof}
We then have the following result by taking sum over $|\gamma^*|=\gamma$.
\begin{proposition}\label{HN,coarse}
    Let $0\neq\gamma\in K_0(\mathcal{D})$, we have the identity
    \[\kappa_\gamma=\sum_{\gamma_1+\dots+\gamma_n=\gamma}\delta^\sigma_{\gamma_1}*\dots*\delta^\sigma_{\gamma_n}.\]
\end{proposition}
However, this identity seems to be so coarse because it has no restriction for $\gamma_1,\dots,\gamma_n$. Recall the case in abelian categories we require that $\phi(\gamma_1)>\phi(\gamma_2)>\dots>\phi(\gamma_n)$. Now we cannot say the same thing because the symbol $\phi(\gamma)$ is not well-defined! Indeed, if $\udv M=\gamma$ with $\phi^+(M)=\phi^-(M)=\phi(M)=\phi$, then $M[2]$ is also $\sigma$-semistable with $\udv M[2]=\gamma$ but $\phi(M[2])=\phi+2\neq\phi$. This inspires us to refine the sets $\mathbf{R}_\gamma,\mathbf{R}_{\gamma^*}^{HN},\mathbf{R}_{\gamma}^{ss}$ and define some well-defined functions related to $\phi$ on them.
\subsection{Refinements}
Let $m\in\mathbb{Z}$, we define the following subsets
\[^m\mathbf{R}_\gamma=\{M\in\mathbf{R}_\gamma|\phi^-(M)\in(2n-1,2n+1]\},\]
\[^m\mathbf{R}_{\gamma^*}^{HN}=\{M\in\mathbf{R}_{\gamma^*}^{HN}|\phi^-(M)\in(2n-1,2n+1]\},\]
\[^m\mathbf{R}_\gamma^{ss}=\{M\in\mathbf{R}_\gamma^{ss}|\phi(M)\in(2n-1,2n+1]\},\]
of $\mathbf{R}_\gamma,\mathbf{R}_{\gamma^*}^{HN}$ and $\mathbf{R}_\gamma^{ss}$, respectively. We then have the following partitions:
\[\mathbf{X}=\bigsqcup_{m\in\mathbb{Z}}{^m}\mathbf{X},\mathbf{X}\in\{\mathbf{R}_\gamma,\mathbf{R}_{\gamma^*}^{HN},\mathbf{R}_\gamma^{ss}\},\]
with bijections between all components:
\[[2n-2m]:{^m}\mathbf{X}\rightarrow{^n}\mathbf{X},\mathbf{X}\in\{\mathbf{R}_\gamma,\mathbf{R}_{\gamma^*}^{HN},\mathbf{R}_\gamma^{ss}\}.\]
We denote by ${^m}\kappa^\sigma_\gamma,{^m}\chi^\sigma_{\gamma^*}$ and ${^m}\delta^\sigma_\gamma$ the characteristic functions supported on ${^m}\mathbf{R}_\gamma,{^m}\mathbf{R}_{\gamma^*}^{HN}$ and ${^m}\mathbf{R}_\gamma^{ss}$, respectively. Hence we have
\[f=\sum_{m\in\mathbb{Z}}{^m}f,f\in\{\kappa_\gamma,\chi^\sigma_{\gamma^*},\delta^\sigma_\gamma\}.\]

We also define a family of well-defined functions on $K_0(\mathcal{D})\backslash\{0\}$:
\[\phi_m:K_0(\mathcal{D})\backslash\{0\}\rightarrow(2m-1,2m+1],\]
which send $\gamma$ to the unique real number $\phi\in(2m-1,2m+1]$ such that $Z(\gamma)=m_\gamma e^{i\pi\phi}$ for some $m_\gamma\in\mathbb{R}_+$.

We can now rewrite \textbf{Lemma \ref{Wang-Chen}} and \textbf{Proposition \ref{HN,coarse}} using the notation we have just defined.

\begin{lemma}[Refinement of \textbf{Lemma \ref{Wang-Chen}}]
    Let $\gamma^*=(\gamma_1,\dots,\gamma_n)$, then
    \[\chi^\sigma_{\gamma^*}=\sum_{\phi_{m_1}(\gamma_1)>\phi_{m_2}(\gamma_2)>\dots>\phi_{m_n(\gamma_n)}}{^{m_1}}\delta^\sigma_{\gamma_1}*{^{m_2}}\delta^\sigma_{\gamma_2}*\dots*{^{m_n}}\delta^\sigma_{\gamma_n}.\]
\end{lemma}
\begin{proof}
    We only need to note that for any $X\in\mathbf{R}_{d^*}^{HN}$ with HN filtration
    \begin{center}
        \begin{tikzcd}[sep=small]
0=X_0 \arrow[rr] &                        & X_1 \arrow[rr] \arrow[ld] &                        & X_2 \arrow[r] \arrow[ld] & \cdots \arrow[r] & X_{n-1} \arrow[rr] &                        & X_n=X \arrow[ld] \\
                 & A_1 \arrow[lu, dashed] &                           & A_2 \arrow[lu, dashed] &                          &                  &                    & A_n \arrow[lu, dashed] &                 
        \end{tikzcd}
    \end{center}
    we have a sequence $m_1,\dots,m_n$ of integers satisfying
    \[A_i\in{^{m_i}}\mathbf{R}_{\gamma_i}^{ss},1\leqslant i\leqslant n.\]
    Then from $\phi(A_i)=\phi_{m_i}(\gamma_i)$ and $\phi(A_1)>\phi(A_2)>\dots>\phi(A_n)$ we deduce $\phi_{m_1}(\gamma_1)>\phi_{m_2}(\gamma_2)>\dots>\phi_{m_n(\gamma_n)}$ by replacing the real numbers one by one. So we conclude the proof from \textbf{Lemma \ref{WC}}.
\end{proof}
\begin{proposition}[Refinement of \textbf{Proposition \ref{HN,coarse}}]\label{HN,fine}
    Let $\gamma\in K_0(\mathcal{D})$, we have the identity
    \[\kappa_\gamma=\sum_{\begin{subarray}{c}
         \\  {\gamma_1+\gamma_2+\dots+\gamma_n=\gamma}\\{\phi_{m_1}(\gamma_1)>\phi_{m_2}(\gamma_2)>\dots>\phi_{m_n(\gamma_n)}}
    \end{subarray}}{^{m_1}}\delta^\sigma_{\gamma_1}*{^{m_2}}\delta^\sigma_{\gamma_2}*\dots*{^{m_n}}\delta^\sigma_{\gamma_n},\]
    and
    \[{^m}\kappa^\sigma_\gamma=\sum_{\begin{subarray}{c}
         \\  {\gamma_1+\gamma_2+\dots+\gamma_n=\gamma},m_n=m\\{\phi_{m_1}(\gamma_1)>\phi_{m_2}(\gamma_2)>\dots>\phi_{m_n(\gamma_n)}}
    \end{subarray}}{^{m_1}}\delta^\sigma_{\gamma_1}*{^{m_2}}\delta^\sigma_{\gamma_2}*\dots*{^{m_n}}\delta^\sigma_{\gamma_n}.\]
\end{proposition}
\subsection{The Right Upper Arrow: Reineke Inversion}
From \textbf{Proposition \ref{HN,fine}}, using the same method of Reineke \cite{Rei2002}, we derive the Reineke inversion in our context within triangulated categories.
\begin{proposition}[Reineke inversion for triangulated categories]\label{Reineke inversion}
    Let $0\neq\gamma\in K_0(\mathcal{D})$. We have
    \[\delta^\sigma_\gamma=\sum_{\begin{subarray}{c}
         \\  \gamma_1+\gamma_2+\dots+\gamma_n=\gamma\\{\phi_{m_j}}(\gamma_1+\dots+\gamma_i)>\phi_m(\gamma),1\leqslant j\leqslant i<n\\
    \end{subarray}}(-1)^n\,{^{m_1}}\kappa^\sigma_{\gamma_1}*\dots*{^{m_n}}\kappa^\sigma_{\gamma_n}.\]
\end{proposition}
\begin{proof}
    By \textbf{Proposition \ref{HN,fine}}, we obtain
    \[{^m}\delta^\sigma_\gamma={^m}\kappa^\sigma_\gamma-\sum_{\begin{subarray}{c}
         \\  n\geqslant2,{\gamma_1+\gamma_2+\dots+\gamma_n=\gamma},m_n=m\\{\phi_{m_1}(\gamma_1)>\phi_{m_2}(\gamma_2)>\dots>\phi_{m_n(\gamma_n)}}
    \end{subarray}}{^{m_1}}\delta^\sigma_{\gamma_1}*{^{m_2}}\delta^\sigma_{\gamma_2}*\dots*{^{m_n}}\delta^\sigma_{\gamma_n}.\]
    Let us consider the set $K_0(\mathcal{D})\backslash\{0\}\times\mathbb{Z}$, with a partial order
    \[(\gamma_1,m_1)\leqslant(\gamma_2,m_2)\iff\phi_{m_1}(\gamma_1)\leqslant\phi_{m_2}(\gamma_2).\]
    We can also define an addition on $K_0(\mathcal{D})\backslash\{0\}\times\mathbb{Z}$:
    \[(\gamma_1,m_1)+(\gamma_2,m_2):=(\gamma_1+\gamma_2,\min(m_1,m_2)),\]
    making $K_0(\mathcal{D})\backslash\{0\}\times\mathbb{Z}$ a commutative subgroup. Now consider the map
    \begin{align*}
        \phi:K_0(\mathcal{D})\backslash\{0\}\times\mathbb{Z}&\rightarrow\mathbb{R}\\
        (\gamma,m)&\mapsto\phi_m(\gamma),
    \end{align*}
    and rewrite the identity we have obtained:
    \[{^m}\delta^\sigma_\gamma={^m}\kappa^\sigma_\gamma-\sum_{\begin{subarray}{c}
         \\  n\geqslant2,(\gamma_1,m_1)+(\gamma_2,m_2)+\dots+(\gamma_n,m_n)=(\gamma,m)\\\phi(\gamma_1,m_1)>\phi(\gamma_2,m_2)>\dots>\phi(\gamma_n,m_n)
    \end{subarray}}{^{m_1}}\delta^\sigma_{\gamma_1}*{^{m_2}}\delta^\sigma_{\gamma_2}*\dots*{^{m_n}}\delta^\sigma_{\gamma_n}.\]
    By the same argument as in abelian categories (see \cite[Theorem 5.1]{Rei2002}), we deduce
    \begin{align*}
        {^m}\delta^\sigma_\gamma&=\sum_{\begin{subarray}{c}
         \\  (\gamma_1,m_1)+(\gamma_2,m_2)+\dots+(\gamma_n,m_n)=(\gamma,m)\\\phi((\gamma_1,m_1)+\dots+(\gamma_i,m_i))>\phi(\gamma,m),1\leqslant i<n\\
    \end{subarray}}(-1)^n\,{^{m_1}}\kappa^\sigma_{\gamma_1}*\dots*{^{m_n}}\kappa^\sigma_{\gamma_n}\\
    &=\sum_{\begin{subarray}{c}
         \\  \gamma_1+\gamma_2+\dots+\gamma_n=\gamma,m_n=m\\{\phi_{m_j}}(\gamma_1+\dots+\gamma_i)>\phi_m(\gamma),1\leqslant j\leqslant i<n\\
    \end{subarray}}(-1)^n\,{^{m_1}}\kappa^\sigma_{\gamma_1}*\dots*{^{m_n}}\kappa^\sigma_{\gamma_n}
    \end{align*}
    By summing over all $m\in\mathbb{Z}$, we finish the proof.
\end{proof}
\subsection{Composition Subalgebras}
As in the case of abelian categories, we now define the composition subalgebra of $\mathcal{H}(\mathcal{D})$. The \textbf{composition subalgebra} $C(\mathcal{D})$ is defined to be the subalgebra of $\mathcal{H}(\mathcal{D})$ generated by the set
\[\{u_{S[n]}|S \text{ is simple in } \phi(0,1], n\in\mathbb{Z}\}.\]
\begin{lemma}
    For any $0\neq\gamma\in\mathcal{D}$ and $m\in\mathbb{Z}$, ${^m}\kappa^\sigma_\gamma,{^m}\delta^\sigma_\gamma\in C(\mathcal{D})$.
\end{lemma}
\begin{proof}
    Let us show that ${^m}\delta^\sigma_\gamma\in C(\mathcal{D})$, and then by \textbf{Proposition \ref{HN,fine}}, we obtain that ${^m}\kappa^\sigma_\gamma\in C(\mathcal{D})$ as well.

    Consider the heart $\mathcal{A}=\mathcal{P}(0,1]$, then the function ${^m}\delta^\sigma_\gamma$ is supported on $\mathcal{A}[2m-1]$ or $\mathcal{A}[2m]$, depending on $\gamma$. Note that $\mathcal{A}[2m-1]$ and $\mathcal{A}[2m]$ are both abelian, so we can use the same method as Reineke \cite[Lemma 4.4]{Rei2002} to deduce the result.
\end{proof}
It is easy to see that $C(\mathcal{D})$ has a natural $K_0(\mathcal{D})$-grading. We denote by $\hat{C}(\mathcal{D})$ the completion of $C(\mathcal{D})$, namely the vector space
\[\prod_{\gamma\in K_0(\mathcal{D})}C(\mathcal{D})_\gamma,\]
with multiplication in the same form as in $C(\mathcal{D})$.

\subsection{Wall-crossing Formulas}
 We can define a comultiplication $\Delta$ on $C(\mathcal{D})$ and $\hat{C}(\mathcal{D})$ via the same form as in the case of abelian categories, namely
 \[\Delta(f)(M,N)=f(M\oplus N),\]
 where we identify $\mathcal{H}(\mathcal{D})\otimes\mathcal{H}(\mathcal{D})$ and a subalgebra of $\mathcal{H}(\mathcal{D}\times\mathcal{D})$ via
 \[(f\otimes g)(M,N)=f(M)g(N).\]
Then we can also consider the set $\hat{\mathfrak{n}}(\mathcal{D})$ of primitive elements and $\hat{N}(\mathcal{D})$ of grouplike elements. As in the case of abelian categories, the exponential map
\[\exp x:=\sum_{n\geqslant0}\frac{x^n}{n!}\]
gives a bijection from $\hat{\mathfrak{n}}(\mathcal{D})$ to $\hat{N}(\mathcal{D})$, with  the inverse
\[\log y=\sum_{n\geqslant1}\frac{(-1)^{n-1}(y-1)^n}{n}.\]

For $\phi\in(2m-1,2m+1]$, we define
\[\Ss^\sigma_\phi=1+\sum_{\phi_m(\gamma)=\phi}{^m}\delta^\sigma_\gamma\in\hat{C}(\mathcal{D}).\]
In other words, $\Ss^\sigma_\phi$ is the characteristic function supported on $\mathcal{P}(\phi)$.
\begin{lemma}\label{SS grouplike}
    For any $\phi\in\mathbb{R}$, $\Ss^\sigma_\phi$ is grouplike and hence belongs to $\hat{N}(\mathcal{D})$.
\end{lemma}
\begin{proof}
    For any $M,N\in\mathcal{D}$, let us consider $\Delta(\Ss^\sigma_\phi)(M,N)$ and $(\Ss^\sigma_\phi\otimes\Ss^\sigma_\phi)(M,N)$. We have
    \[\Delta(\Ss^\sigma_\phi)(M,N)=\Ss^\sigma_\phi(M\oplus N)=\begin{cases}
        1,&M\oplus N\in\mathcal{P}(\phi),\\
        0,&\text{otherwise},
    \end{cases}\]
    and
    \[(\Ss^\sigma_\phi\otimes\Ss^\sigma_\phi)(M,N)=\Ss^\sigma_\phi(M)\Ss^\sigma_\phi(N)=\begin{cases}
        1,&M,N\in\mathcal{P}(\phi),\\
        0,&\text{otherwise}.
    \end{cases}\]
    We finish the proof by \textbf{Lemma \ref{semistable vs direct sum}}.
\end{proof}
It is easy to see that the function $1_\mathcal{D}$, which takes the value $1$ on every isomorphism class of objects in $\mathcal{D}$, is a grouplike element, and we have the following identity.

\begin{theorem}[Wall-crossing formula for triangulated categories]
    \[1_\mathcal{D}=\prod^\leftarrow_{\phi\in\mathbb{R}}\Ss^\sigma_\phi.\]
\end{theorem}
\begin{proof}
    Let us calculate from the right-hand side.
    \begin{align*}
        \prod^\leftarrow_{\phi\in\mathbb{R}}\Ss^\sigma_\phi&=\prod^\leftarrow_{\phi\in\mathbb{R}}\left(1+\sum_{\phi_m(\gamma)=\phi}{^m}\delta^\sigma_\gamma\right)\\
        &=1+\sum_{\phi_1>\phi_2>\dots>\phi_s}\left(\sum_{\phi_{m_1}(\gamma_1)=\phi_1}{^{m_1}}\delta^\sigma_{\gamma_1}\right)*\dots*\left(\sum_{\phi_{m_s}(\gamma_s)=\phi_s}{^{m_s}}\delta^\sigma_{\gamma_s}\right)\\
        &=1+\sum_{\phi_{m_1}(\gamma_1)>\phi_{m_2}(\gamma_2)>\dots>\phi_{m_s}(\gamma_s)}{^{m_1}}\delta^\sigma_{\gamma_1}*{^{m_2}}\delta^\sigma_{\gamma_2}*\dots*{^{m_s}}\delta^\sigma_{\gamma_s}\\
        &=1+\sum_{\gamma\in K_0(\mathcal{D})}\sum_{\begin{subarray}{c}
         \\{\phi_{m_1}(\gamma_1)>\phi_{m_2}(\gamma_2)>\dots>\phi_{m_s(\gamma_s)}}\\
        \gamma_1+\gamma_2+\dots+\gamma_s=\gamma
    \end{subarray}}{^{m_1}}\delta^\sigma_{\gamma_1}*{^{m_2}}\delta^\sigma_{\gamma_2}*\dots*{^{m_n}}\delta^\sigma_{\gamma_s}\\
    &=1+\sum_{\gamma\in K_0(\mathcal{D})}\kappa_\gamma\text{ (by \textbf{Proposition \ref{HN,fine}})}\\
    &=1_\mathcal{D}.
    \end{align*}
\end{proof}
\subsection{Joyce's Elements and the Left Part of the Diagram}
In this subsection we generalize the context of Joyce's element. Recall in abelian categories, the Joyce's element is defined to be
\[\epsilon^Z_\gamma=\sum_{\begin{subarray}{c}
    \\ \gamma_1+\dots+\gamma_n=\gamma\\ \phi(\gamma_i)=\phi(\gamma),1\leqslant i\leqslant n
\end{subarray}}\frac{(-1)^{n-1}}{n}\delta^Z_{\gamma_1}*\dots*\delta^Z_{\gamma_n}\in C(\mathcal{A}).\]
We have shown that it lies in $\mathfrak{n}(\mathcal{A})$, so it is supported on indecomposable objects. See Subsection 3.5. Now we consider the element in the same form:
\[\tilde{\epsilon^\sigma_\gamma}=\sum_{\begin{subarray}{c}
        \\ \gamma_1+\dots+\gamma_n=\gamma\\
        Z(\gamma_i)\in\mathbb{R}_+Z(\gamma),1\leqslant i\leqslant n
    \end{subarray}}\frac{(-1)^{n-1}}{n}\delta^\sigma_{\gamma_1}*\dots*\delta^\sigma_{\gamma_n}.\]
However, the following example implies that it might not be the correct generalization of Joyce's element, since it takes nonzero value on some decomposable objects.
\begin{example}
    Let $Q$ be an acyclic quiver and $\mathcal{D}=D^b(Q)$, so $\mathcal{D}$ is heredietary and satisfies our two conditions. Let $S$ be the simple representation of $Q$ associated to any vertex and let $\gamma=2\udv S$. Then
    \begin{align*}
        \tilde{\epsilon}^\sigma_\gamma&=\delta^\sigma_\gamma-\frac12(\delta^\sigma_{\udv S}*\delta^\sigma_{\udv S}+\delta^\sigma_{\udv S}*\delta^\sigma_{\udv S})\\
        &=\delta^\sigma_\gamma-\delta^\sigma_{\udv S}\\
        &=\sum_{m\in\mathbb{Z}}u_{S[2m]\oplus S[2m]}-\sum_{m,n\in\mathbb{Z}}u_{S[2m]\oplus S[2n]}\\
        &=-\sum_{m\neq n\in\mathbb{Z}}u_{S[2m]\oplus S[2n]}.
        \end{align*}
    So $\tilde{\epsilon}^\sigma_\gamma$ takes the value $-1$ on $S[2m]\oplus S[2n],m\neq n$, which are decomposable objects.
\end{example}
Now we provide a possible way to generalize Joyce's elements. Recall in Subsection 3.5 we have calculated that
\[\log\Ss^Z_\phi=\sum_{\gamma\in K_0(\mathcal{A})^+,\phi(\gamma)=\phi}\epsilon^Z_\gamma\]
in the context of abelian categories. So let us calculate the logarithm of $\Ss^\sigma_\phi$:
\begin{align*}
    \log\Ss^\sigma_\phi&=\log\left(1+\sum_{\phi_m(\gamma)=\phi}{^m}\delta^\sigma_\gamma\right)\\
    &=\sum_{n\geqslant1}\frac{(-1)^{n-1}}{n}\left(\sum_{\phi_m(\gamma)=\phi}{^m}\delta^\sigma_\gamma\right)^n\\
    &=\sum_{n\geqslant1}\frac{(-1)^{n-1}}{n}\sum_{\phi_m(\gamma_1)=\phi_m(\gamma_n)=\phi}{^m}\delta^\sigma_{\gamma_1}*\dots*{^m}\delta^\sigma_{\gamma_n}.
\end{align*}
If we set
\[{^m}\epsilon^\sigma_\gamma=\sum_{\begin{subarray}{c}
    \\ \gamma_1+\dots+\gamma_n=\gamma\\
   \phi_m(\gamma_1)=\dots=\phi_m(\gamma_n)=\phi_m(\gamma) 
\end{subarray}}\frac{(-1)^{n-1}}{n}{^m}\delta^\sigma_{\gamma_1}*\dots*{^m}\delta^\sigma_{\gamma_n},\]
then we have
\[\log\Ss^\sigma_\phi=\sum_{\phi_m(\gamma)=\phi}{^m}\epsilon^\sigma_\gamma.\]
Let
\[\epsilon^\sigma_\gamma=\sum_{m\in \mathbb{Z}}{^m}\epsilon^\sigma_\gamma\in\hat{C}(\mathcal{D}).\]
By \textbf{Lemma \ref{SS grouplike}}, $\Ss^\sigma_\phi\in\hat{N}(\mathcal{D})$, hence its logarithm lies in $\hat{\mathfrak{n}}(\mathcal{D})$ and is primitive. Hence ${^m}\epsilon^\sigma_\gamma$, and then $\epsilon^\sigma_\gamma$ are all primitive, or equivalently, supported on indecomposable objects.

When $\mathcal{D}$ is hereditary, we have
\begin{theorem}\label{Left part}
Let $\ell$ be a ray in the complex plane with its vertex at the origin and argument $\phi/\pi\in(-1,1]$, then
    \[\exp\left(\sum_{Z(\gamma)\in\ell}\epsilon^\sigma_\gamma\right)=\prod_{m\in\mathbb{Z}}\Ss^\sigma_{\phi+2m}.\]
\end{theorem}
\begin{proof}
    We first note that
    \begin{align*}
        \sum_{Z(\gamma)\in\ell}\epsilon^\sigma_\gamma&=\sum_{Z(\gamma)\in\ell}\sum_{m\in\mathbb{Z}}{^m}\epsilon^\sigma_\gamma\\
        &=\sum_{\phi_0(\gamma)=\phi}\sum_{m\in\mathbb{Z}}{^m}\epsilon^\sigma_\gamma\\
        &=\sum_{m\in\mathbb{Z}}\sum_{\phi_m(\gamma)=\phi+2m}{^m}\epsilon^\sigma_\gamma.
    \end{align*}
    By the heredity of $\mathcal{D}$, when $m\neq n$, we have
    \[[u_{A[2m]},u_{B[2n]}]=0\]
    for any $A,B\in\mathcal{D}$ indecomposable with $\phi_0(\udv A)=\phi_0(\udv B)$. Hence, ${^m}\epsilon^\sigma_\gamma$ and ${^n}\epsilon^\sigma_\gamma$ are commutative. Therefore,
    \begin{align*}
        \exp\left(\sum_{Z(\gamma)\in\ell}\epsilon^\sigma_\gamma\right)&=\exp\left(\sum_{m\in\mathbb{Z}}\sum_{\phi_m(\gamma)=\phi+2m}{^m}\epsilon^\sigma_\gamma\right)\\
        &=\prod_{m\in\mathbb{Z}}\exp\left(\sum_{\phi_m(\gamma)=\phi+2m}{^m}\epsilon^\sigma_\gamma\right)\\
        &=\prod_{m\in\mathbb{Z}}\Ss^\sigma_{\phi+2m}.
    \end{align*}
\end{proof}

\subsection{Summary: Reconstruct the Diagram}
Now we can draw the diagram similar to the one in \cite{BT2012}.
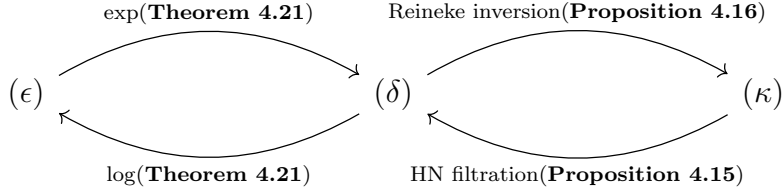
\begin{center}
\begin{tikzcd}
(\epsilon) \arrow[rrrr, "\exp(\textbf{Theorem \ref{Left part}})", bend left] & & &  & (\delta) \arrow[llll, "\log(\textbf{Theorem \ref{Left part}})", bend left] \arrow[rrrr, "\text{Reineke inversion}(\textbf{Proposition \ref{Reineke inversion}})", bend left] & & &  & (\kappa) \arrow[llll, "{{\text{HN filtration} (\textbf{Proposition \ref{HN,fine}}})}", bend left]
\end{tikzcd}
\end{center}
\section{Examples for Wall-crossing Formulas: The Case $A_2$}
Let us continue our discussion on the quiver $A_2$ in \textbf{Example \ref{A2}} and \textbf{Example \ref{A2,abelian}}.

Now consider the category $\mathcal{D}=D^b(\rep A_2)$. By \cite{Hap1988}, every object in $\mathcal{D}$ can be factorized as
\[M=S_1[n_1]^{\oplus p_1}\oplus\dots\oplus S_1[n_k]^{\oplus p_k}\oplus S_2[m_1]^{\oplus q_1}\oplus\dots\oplus S_2[m_l]^{\oplus q_l}\oplus P_1[s_1]^{\oplus r_1}\oplus\dots\oplus P_1[s_t]^{\oplus r_t}.\]
So by \textbf{Lemma \ref{semistable vs direct sum}} and \textbf{Lemma \ref{stability under shift}}, the stability of $M$ depends on the stabilities of $S_1,S_2$ and $P_1$, and the relations between their phases. In other words, we only need to be concerned about the stabilities of these three indecomposable objects.

If none of the three objects is $\sigma$-semistable, then there exists no $\sigma$-semistable object in $\mathcal{D}$. That contradicts the existence of HN filtration. If only one of the three objects, say $S_1$, is $\sigma$-semistable, then the wall-crossing formula yields
\[1_\mathcal{D}=\dots*\Phi_{S_1[n+1]}*\Phi_{S_1[n]}*\Phi_{S_1[n-1]}*\cdots.\]
However, the identity never holds, since the right-hand side only supported on objects have direct summands of the form $S_1[n]$. As a result, at least two of the three objects $S_1,S_2$ and $P_1$ are $\sigma$-semistable.

\textbf{Case 1: $S_1$ and $S_2$ are $\sigma$-semistable but $P_1$ is not.}By \textbf{Lemma \ref{phase increasing}}, the triangle
\[S_2\rightarrow P_1\rightarrow S_1\rightarrow S_2[1]\]
requires that $\phi(S_2)>\phi(S_1)$. There exists a unique nonnegative integer $n$ such that $\phi(S_1)+n<\phi(S_2)\leqslant\phi(S_1)+n+1$. The wall-crossing formula is now written as
\begin{equation}\label{1}
    1_\mathcal{D}=\dots*\Phi_{S_2}*\Phi_{S_1[n]}*\Phi_{S_2[-1]}*\Phi_{S_1[n-1]}*\cdots.
\end{equation}

Now we consider the heart $\mathcal{H}$ of $\sigma$. Up to a finite time of shifts, we may assume that $S_2\in\mathcal{H}=\phi(0,1]$. Thus, by the inequality $\phi(S_1)+n<\phi(S_2)\leqslant\phi(S_1)+n+1$, we know that one of $S_1[n]$ and $S_1[n+1]$ is in $\mathcal{H}$.

If $S_1[n+1]\in\mathcal{H}$, then $\mathcal{H}=\add(S_2,S_1[n+1])$, which is the heart of the bounded t-structure of $\mathcal{D}$ corresponding to the silting complex (see \cite{KY2014})
\[\dots\rightarrow0\rightarrow P_1\rightarrow0\rightarrow\dots\rightarrow0\rightarrow S_2\rightarrow0\rightarrow\cdots\]
in $K^b(\proj\mathcal{A})$, where $P_1$ is the $(-n-1)$-th component and $S_2$ is the $0$-th component. The corresponding stability function, in the sense of \textbf{Theorem \ref{heart}}, satisfies that $\phi(S_2)\leqslant\phi(S_1[n+1])$.

If $S_1[n]\in\mathcal{H}$ with $n\geqslant1$, then $\mathcal{H}=\add(S_2,S_1[n])$, which is the heart of the bounded t-structure of $\mathcal{D}$ corresponding to the silting complex
\[\dots\rightarrow0\rightarrow P_1\rightarrow0\rightarrow\dots\rightarrow0\rightarrow S_2\rightarrow0\rightarrow\cdots\]
in $K^b(\proj\mathcal{A})$, where $P_1$ is the $(-n)$-th component and $S_2$ is the $0$-th component. The corresponding stability function satisfies that $\phi(S_2)>\phi(S_1[n])$.

If $S_1\in\mathcal{H}$, then by the extension-closedness of $\mathcal{H}$, $P_1$ lies in $\mathcal{H}$ as well. In this case, the heart $H$ is simply $\mathcal{A}$, with the corresponding stability function satisfying $\phi(S_2)>\phi(S_1)$.

\textbf{Case 2: $S_1$ and $P_1$ are $\sigma$-semistable but $S_2$ is not.} Now we consider the triangle
\[S_1\rightarrow S_2[1]\rightarrow P_1[1]\rightarrow S_1[1].\]
Similarly to Case 1 we have $\phi(S_1)>\phi(P_1[1])=\phi(P_1)+1$. There exists a unique nonnegative integer $n$ such that $\phi(P_1)+n+1<\phi(S_1)\leqslant\phi(P_1)+n+2$. The wall-crossing formula is now written as
\begin{equation}\label{2}
    1_\mathcal{D}=\dots*\Phi_{S_1}*\Phi_{P_1[n+1]}*\Phi_{S_1[-1]}*\Phi_{P_1[n]}*\cdots.
\end{equation}

Again, we may assume that $S_1\in\mathcal{H}=\phi(0,1]$, then one of $P_1[n+1]$ and $P_1[n+2]$ lies in $\mathcal{H}$. 

If $P_1[n+2]\in\mathcal{H}$, then $\mathcal{H}=\add(S_1,P_1[n+2])$, which is the heart of the bounded t-structure of $\mathcal{D}$ corresponding to the silting complex
\[\dots\rightarrow0\rightarrow S_2\rightarrow0\rightarrow\dots\rightarrow0\rightarrow S_2\hookrightarrow P_1\rightarrow0\rightarrow\cdots\]
in $K^b(\proj\mathcal{A})$, where $P_1$ is the $0$-th component and the two $S_2$'s are the $-1$-st and $(-n-2)$-th components, respectively. The corresponding stability function satisfies that $\phi(S_1)\leqslant\phi(P_1[n+2])$.

If $P_1[n+1]\in\mathcal{H}$ with $n\geqslant1$, then $\mathcal{H}=\add(S_1,P_1[n+1])$, which is the heart of the bounded t-structure of $\mathcal{D}$ corresponding to the silting complex
\[\dots\rightarrow0\rightarrow S_2\rightarrow0\rightarrow\dots\rightarrow0\rightarrow S_2\hookrightarrow P_1\rightarrow0\rightarrow\cdots\]
in $K^b(\proj\mathcal{A})$, where $P_1$ is the $0$-th component and the two $S_2$'s are the $-1$-st and $(-n-1)$-th components, respectively. The corresponding stability function satisfies that $\phi(S_1)>\phi(P_1[n+1])$.

If $P_1[1]\in\mathcal{H}$, then by the extension-closedness of $\mathcal{H}$, $S_2[1]$ lies in $\mathcal{H}$ as well. In this case, the heart $H$ is $\add(S_1,S_2[1],P_1[1])$, with the corresponding stability function satisfying $\phi(S_1)>\phi(P_1[1])=\phi(P_1)+1$.

\textbf{Case 3: $P_1$ and $S_2$ are $\sigma$-semistable but $S_1$ is not.} Now we consider the triangle
\[P_1\rightarrow S_1\rightarrow S_2[1]\rightarrow P_1[1].\]
Similarly to Cases 1 and 2 we have $\phi(P_1)>\phi(S_2[1])=\phi(S_2)+1$. There exists a unique nonnegative integer $n$ such that $\phi(S_2)+n+1<\phi(P_1)\leqslant\phi(S_2)+n+2$. The wall-crossing formula is now written as
\begin{equation}\label{3}
    1_\mathcal{D}=\dots*\Phi_{P_1}*\Phi_{S_2[n+1]}*\Phi_{P_1[-1]}*\Phi_{S_2[n]}*\cdots.
\end{equation}

Again, we may assume that $P_1\in\mathcal{H}=\phi(0,1]$, then one of $S_1[n+1]$ and $S_1[n+2]$ lies in $\mathcal{H}$. 

If $S_2[n+2]\in\mathcal{H}$, then $\mathcal{H}=\add(P_1,S_2[n+2])$, which is the heart of the bounded t-structure of $\mathcal{D}$ corresponding to the silting complex
\[\dots\rightarrow0\rightarrow S_2\hookrightarrow P_1\rightarrow0\rightarrow\dots\rightarrow0\rightarrow P_1\rightarrow0\rightarrow\cdots\]
in $K^b(\proj\mathcal{A})$, where $S_2$ is the $(-n-2)$-th component and the two $P_1$'s are the $(-n-1)$-th and $0$-th components, respectively. The corresponding stability function satisfies that $\phi(P_1)\leqslant\phi(S_2[n+2])$.

If $S_2[n+1]\in\mathcal{H}$ with $n\geqslant1$, then $\mathcal{H}=\add(S_1,P_1[n+1])$, which is the heart of the bounded t-structure of $\mathcal{D}$ corresponding to the silting complex
\[\dots\rightarrow0\rightarrow S_2\hookrightarrow P_1\rightarrow0\rightarrow\dots\rightarrow0\rightarrow P_1\rightarrow0\rightarrow\cdots\]
in $K^b(\proj\mathcal{A})$, where $S_2$ is the $(-n-1)$-th component and the two $P_1$'s are the $(-n)$-th and $0$-th components, respectively. The corresponding stability function satisfies that $\phi(P_1)>\phi(S_2[n+1])$.

If $S_2[1]\in\mathcal{H}$, then by the extension-closedness of $\mathcal{H}$, $P_1$ lies in $\mathcal{H}$ as well. In this case, the heart $H$ is $\add(P_1,S_1,S_2[1])$, with the corresponding stability function satisfying $\phi(S_1)>\phi(P_1[1])=\phi(P_1)+1$.

\textbf{Case 4: $S_1,S_2,P_1$ are all $\sigma$-semistable.} Now we must have $\phi(S_2)\leqslant\phi(P_1)\leqslant\phi(S_1)\leqslant\phi(S_2)+1$ by \textbf{Lemma \ref{phase increasing}}. The wall-crossing formula is now written as
\begin{equation}\label{4}
    1_\mathcal{D}=\dots*\Phi_{S_2[n+1]}*\Phi_{S_1[n]}*\Phi_{P_1[n]}*\Phi_{S_2[n]}*\Phi_{S_1[n-1]}*\cdots.
\end{equation}

The heart of $\sigma$ has three possibilities (up to shifting): $\mathcal{A},\add(P_1,S_1,S_2[1])$ and $\add(S_1,S_2[1],P_1[1])$. It depends on the relationship between values $\phi(S_2),\phi(P_1),\phi(S_1)$ and the interval $(0,1]$.

We summarize all the results above in the following table:

\begin{longtable}{@{}|c|c|c|c|@{}}
\hline
$\sigma$-semistable    & \multirow{2}{*}{heart}    & \multirow{2}{*}{phases}            & wall-crossing        \\
objects                        &                           &                                    & formula              \\* \hline
\endfirsthead
\multicolumn{4}{c}%
{{\bfseries Table \thetable\ continued from previous page}} \\
\endhead
\multirow{3}{*}{$S_1,S_2$}     & $\add(S_2,S_1[n+1])$      & $\phi(S_2)\leqslant\phi(S_1[n+1])$ & \multirow{2}{*}{(\ref{1})} \\* \cline{2-3}
                               & $\add(S_2,S_1[n])$        & $\phi(S_2)>\phi(S_1[n])$           &                      \\* \cline{2-4} 
                               & $\mathcal{A}$             & $\phi(S_2)>\phi(S_1)$              & (\ref{1}) with $n=0$       \\* \hline
\multirow{3}{*}{$P_1,S_1$}     & $\add(S_1,P_1[n+2])$      & $\phi(S_1)\leqslant\phi(P_1[n+2])$ & \multirow{2}{*}{(\ref{2})} \\* \cline{2-3}
                               & $\add(S_1,P_1[n+1])$      & $\phi(S_1)>\phi(P_1[n+1])$         &                      \\* \cline{2-4} 
                               & $\add(S_1,S_2[1],P_1[1])$ & $\phi(S_1)>\phi(P_1[1])$           & (\ref{2}) with $n=0$       \\* \hline
\multirow{3}{*}{$S_2,P_1$}     & $\add(P_1,S_2[n+2])$      & $\phi(P_1)\leqslant\phi(S_2[n+2])$ & \multirow{2}{*}{(\ref{3})} \\* \cline{2-3}
                               & $\add(P_1,S_2[n+1])$      & $\phi(P_1)>\phi(S_2[n+1])$         &                      \\* \cline{2-4} 
                               & $\add(P_1,S_1,S_2[1])$    & $\phi(P_1)>\phi(S_2[1])$           & (\ref{3}) with $n=0$       \\* \hline
\multirow{3}{*}{$S_2,P_1,S_1$} & $\mathcal{A}$             & $\phi(S_2)\leqslant\phi(S_1)$      & \multirow{3}{*}{(\ref{4})} \\* \cline{2-3}
                               & $\add(S_1,S_2[1],P_1[1])$ & $\phi(S_1)\leqslant\phi(P_1[1])$   &                      \\* \cline{2-3}
                               & $\add(P_1,S_1,S_2[1])$    & $\phi(P_1)\leqslant\phi(S_2[1])$   &                      \\* \hline
\end{longtable}

\section*{Acknowledgments}
This work was supported by NSF of China (No. 12371036).

\end{document}